\documentclass[11pt]{article}
\usepackage{subcaption}
\usepackage{authblk}
\usepackage{hyperref}
\usepackage{forest, float}

\usepackage{amssymb,amsthm,amsmath,mathrsfs,fullpage,multirow}
\usepackage{tikz}
\usetikzlibrary{matrix,decorations.pathreplacing}
\usepackage{ytableau}
\usetikzlibrary{arrows.meta}

\tikzset{
  solid node/.style={circle,draw,inner sep=1.2,fill=black},
  hollow node/.style={circle,draw,inner sep=1.2},
  left label/.style={above left,midway},
  right label/.style={above right,midway}
}

\makeatletter
\def\keywords{\xdef\@thefnmark{}\@footnotetext}
\makeatother

\DeclareMathOperator{\des}{des}

\renewcommand\S{\mathcal S}

\newcommand\C{\mathcal C}

\renewcommand\O{\mathcal O}

\newtheorem{theorem}{Theorem}[section]
\newtheorem{lemma}[theorem]{Lemma}
\newtheorem{proposition}[theorem]{Proposition}
\newtheorem{corollary}[theorem]{Corollary}
\newtheorem{conjecture}[theorem]{Conjecture}
\theoremstyle{definition}
\newtheorem{definition}[theorem]{Definition}
  \newtheorem{question}[theorem]{Question}
\newtheorem{remark}[theorem]{Remark}
\newtheorem{example}[theorem]{Example}

\newcommand{\thistheoremname}{}
\newtheorem*{genericthm*}{\thistheoremname}
\newenvironment{namedthm*}[1]
  {\renewcommand{\thistheoremname}{#1}%
   \begin{genericthm*}}
  {\end{genericthm*}}

\usepackage{enumerate}

\title{Pattern avoidance in canon permutations}
\author[ ]{Robert Laudone}
\affil[ ]{{\small Department of Mathematics, United States Naval Academy, Annapolis, MD, 21402}}
\affil[ ]{{\small Email: laudone@usna.edu }}

\date{}

\begin{document}


\maketitle

\begin{abstract}
    A canon permutation is a $k$-regular word over $[n]$ in which, for each $j$, the $j$-th copies of the letters form the same permutation $\sigma$. These were introduced by Elizalde as a generalization of nonnesting multipermutations, which are the case $k = 2$. We study classical pattern avoidance in them for arbitrary $k$. We show that avoiding any one of $112$, $122$, $211$ or $221$ is counted by the $k$-Catalan numbers $\frac{1}{n}\binom{kn}{n-1}$. We enumerate the classes obtained by forbidding one of these together with any $\tau \in \mathcal{S}_3$, and we give a bijection with $k$-ary trees that we use to generalize a theorem of Gabriel, Peske, Pudwell and Tay. We then show that avoiding a set of patterns closed under relabeling reduces, up to a factor of $n!$, to avoidance in $k$-regular lattice words. We use this to enumerate the canon permutations avoiding some natural generalizations of the nonnesting and noncrossing patterns, as well as the family $\{1^a21^b, 2^a12^b\}$. We close with several conjectures and questions.
\end{abstract}

\section{Introduction}

Let $\S_n$ be the set of permutations of $[n] = \{1,\dots,n\}$ and let $\mathcal{M}^k_n = \{1^k,\dots,n^k\}$ be the multiset with $k$ copies of each number in $[n]$. We denote by $\S_n^k$ the permutations of $\mathcal{M}_n^k$; these are also sometimes called $k$-regular words over $[n]$. By a \emph{pattern} we mean any word $\sigma = \sigma_1\cdots\sigma_m$ over the alphabet $[m]$ in which each of $1,\dots,m$ occurs at least once, and we write $\mathcal{P}$ for the set of all patterns. Note that a pattern need not be regular, for example $112 \in \mathcal{P}$. We say that a permutation $\pi = \pi_1\ldots\pi_{nk} \in \S_n^k$ contains a pattern $\sigma = \sigma_1 \dots \sigma_m$ if there are indices $1 \leq i_1 < \cdots < i_m \leq nk$ where $\pi_{i_1}\pi_{i_2}\ldots\pi_{i_m}$ has the same relative order as $\sigma$. For example, if $\sigma = 112$, then $3{\bf 22} 1{\bf3}1$ contains $112$ in the bold positions. We say $\pi$ \emph{avoids} $\sigma$ if it does not contain it. We denote by $\S_n^k(\sigma)$ the multipermutations that avoid $\sigma$.

The case $k=2$ has been extensively studied. This began with Stirling permutations $\S_n^2(212)$, introduced by Gessel and Stanley \cite{GS} in 1978, whose descent polynomials give a combinatorial interpretation of the numerators for generating functions associated to Stirling numbers. Later, two variants were introduced: the quasi-Stirling permutations $\S_n^2(1212,2121)$ of \cite{AGPS19}, which avoid crossings and so are now also called noncrossing permutations, and the nonnesting permutations $\S_n^2(1221,2112)$, defined in \cite{E23}. Additional avoidance conditions in both cases were considered in \cite{AGPS19,AL25,EL25}. Other work in this area includes restricted Stirling permutations \cite{CMM}, block patterns in Stirling permutations \cite{RW}, and descents on quasi-Stirling permutations \cite{E21}.

Increasing $k$ leads to many other interesting results, where the $k$-Catalan numbers, sometimes called the Fuss-Catalan numbers, $\C_{k,n} = \frac1n \binom{kn}{n-1}$ appear repeatedly. When $k = 2$ these are the usual Catalan numbers $C_n$. For example, when $k \ge 2$,  Kuba and Panholzer \cite{KP} proved $|\S_n^{k-1}(212,312)| = \C_{k,n}$, Defant and Kravitz \cite{DK20} proved $|\S_n^{k-1}(221,231)| = \C_{k,n}$, and Williams \cite{W23} proved $|\S_n^{k-1}(112,123)| = \C_{k,n}$.

In \cite{E23}, Elizalde defined canon permutations as one of the most natural generalizations of nonnesting permutations for larger values of $k$. We say that $\pi \in \S_n^k$ is a canon permutation if there exists $\sigma \in \S_n$ such that for every $j \in [k]$, the subsequence of $\pi$ formed by the $j$-th copy of each entry (read from the left) is $\sigma$. We denote by $\mathcal{C}_n^k$ the set of \emph{canon permutations} of $\mathcal{S}_n^k$, by $\mathcal{C}_n^{k,\sigma}$ the set of canon permutations where the $j$-th copies of each entry form the permutation $\sigma$, and by $\mathcal{C}_n^{k}(\Lambda)$ the set of canon permutations avoiding all the elements of the set $\Lambda = \{\sigma_1,\dots,\sigma_r\}$. We let $c_n^{k,\sigma}(\Lambda) = |\C_n^{k,\sigma}(\Lambda)|$ and $c_n^k(\Lambda) = |\C_n^k(\Lambda)|$. When $k = 2$, these are precisely pattern avoiding nonnesting permutations, which were studied in \cite{EL25}. 

Elizalde's work on canon permutations in \cite{E23, E25} largely centered on the distribution of their descents. He showed their descent polynomial has a surprising product structure. Beck and Deligeorgaki \cite{BD24} recovered and generalized Elizalde's results using Stanley's $(P,\omega)$-partitions \cite{Sta11,Sta72}. They did this by realizing canon permutations as linear extensions of the product poset $[k] \times [n]$ up to suitable labeling. We return to this viewpoint in Section \ref{sec: symmetry}, and our results imply the descent polynomial statement of Theorem \ref{thm: 1a21b-descents}, which motivates further questions about canon descent polynomials with additional avoidance conditions. One of our first results further extends the family of $k$-regular words with a $k$-Catalan structure:

\begin{namedthm*}{Theorem \ref{thm: 112}}
    For $n \geq 1$ and $k \geq 2$,
    \[
    |\C_n^k(\tau)| = \C_{k,n}
    \]
    for $\tau \in \{112,122,211,221\}$.
\end{namedthm*}

We then investigate how other avoidance conditions interact with $k$-regular non-Stirling canon permutations for any $k \geq 2$. A summary of our results can be found in Table \ref{tab:results}. We note that up to symmetry the first five rows of the table along with Corollary \ref{cor: 112-123} and Remark \ref{rem: 321} suffice to enumerate $\C_n^k(\sigma,\tau)$ for any $\sigma \in \{112,122,211,221\}$ and $\tau \in \S_3 \cup \{\varnothing\}$. We restrict to the non-Stirling case because Proposition \ref{prop: 212} shows that the Stirling canon permutations lose the multidimensional structure that makes the others interesting. In Section \ref{sec: trees} we relate avoidance in $k$-ary trees to canon permutations avoiding $112$ and other patterns. In Section \ref{sec: symmetry}, we define and study symmetric avoidance in canon permutations, together with its connection to $k$-regular lattice words and to linear extensions of posets refining the Young diagram poset of shape $(k^n)$. We conclude in Section \ref{sec: further-study} with some open questions and directions for further study.

\renewcommand{\arraystretch}{1.3}
\begin{table}[H]
    \centering
    \begin{tabular}{|c|c|c|} 
    \hline
         $\Lambda$ & $c_n^k(\Lambda)$ & \text{Theorem}\\
         \hline \hline
         $\{112\}$ & \multirow{1}{*}{$\C_{k,n}$} &\multirow{1}{*}{Theorem \ref{thm: 112}}\\   \hline
         $\{112, 132\}$ & \multirow{2}{*}{$k^{n-1}$} &\multirow{2}{*}{Theorem \ref{thm: 112-132}}\\  \cline{1-1}
         $\{112, 213\}$ &&\\ \hline
         $\{112, 312\}$ & \multirow{2}{*}{$k$} &\multirow{2}{*}{Theorem \ref{thm: 112-312}}\\  \cline{1-1}
         $\{112, 231\}$ &&\\ \hline
         $\{1221,2112\}$ & $n! \, 2^{n-1}$ \, if \, $k \ge 3$ &Proposition \ref{prop: 1221-2112}\\ \hline
         \multirow{5}{5em}{\centering $\{12^j1,21^j2\}$ for $j \geq 2$} & $n!\, \binom{2j-2}{j-1}^{n-1}$ \, if \, $2(j-1) < k$ & Theorem \ref{thm: 12j1-2(j-1)<k} \\ \cline{2-3}
         &$n! \,T_{n-1}(k,j-1)$ \, if \, $2(j-1) = k$&Theorem \ref{thm: 12j1-k=2(j-1)}\\ \cline{2-3}
         &$2 \,\left[\binom{2j-2}{j-1} - \binom{2j-2}{2j-k-3} \right]$ \, if \,$n=2$ and $2 \le j \le k$&Proposition \ref{prop: 121-n=2}\\ \cline{2-3}
         &Open for $n \geq 3$ if $j < k < 2(j-1)$&Question \ref{Q: 12j1}\\ \cline{2-3}
         &$n!\, C_n^{(k)}$ \, if \,  $j \geq k$&Theorem \ref{thm: 12j1-j>=k}\\ \hline
         $\{1212,2121\}$ & $n!$ &Proposition \ref{prop: 1212-2121}\\ \hline
         $\{(12)^j,(21)^j\}$  for $j \ge 2$ & $2 \sum_{i=1}^{j-1} N(k,i)$ for $n = 2$, $k \ge 1$ &Theorem \ref{thm: 12j-21j}\\ \hline
         $\{1^j2^{k+1 - j}\}$& $C_n^{(k)}$ &Proposition \ref{prop: standardTab} \\ \hline
         $\{1^a21^b,2^a12^b\}$ for $a,b \ge 1$& $n! \, C_n^{(\min(a+b-1,k))}$ &Theorem \ref{thm: 1a21b-2a12b} \\ \hline
    \end{tabular}
    \caption{Summary of enumeration results for avoidance in canon permutations. $\C_{k,n}$ is the $k$-Catalan number of size $n$, $T_{n}(s,t)$ is the number of solutions to the $(s,t)$-tennis ball problem, $C_n^{(k)}$ are the $k$-dimensional Catalan numbers \cite{Su04}, $N(k,i)$ are the Narayana numbers.}
    \label{tab:results}
\end{table}

\begin{remark}
    \sloppy Several of the sequences from Table \ref{tab:results} appear in OEIS. The Fuss-Catalan numbers $\C_{k,n}$ form the columns of \cite[A062993]{OEIS}, the $k$-dimensional Catalan numbers $C_n^{(k)}$ are \cite[A000108, A005789, A005790, A005791]{OEIS} for $k = 2,3,4,5$ and $n! \, 2^{n-1}$ is \cite[A002866]{OEIS}.
\end{remark}

\section{Background and Notation} \label{sec: background}
We will represent canon permutations using the framework developed in \cite{E25}. A canon permutation $\pi \in \mathcal{C}_n^k$ corresponds to a pair consisting of a permutation $\sigma_\pi \in \S_n$ and a Young tableau of shape $(k^n)$, denoted by $T_\pi$. The tableau records the position of each entry of the permutation $\sigma_\pi$. For example,
\begin{center}
\ytableausetup{centertableaux}

   \begin{ytableau}
       \none[2] &1 &3 &6\\
       \none[1] &2&5&7\\
       \none[3] &4&8&9
   \end{ytableau}
   \end{center}

 \noindent corresponds to the canon permutation $212312133 \in \mathcal{C}_3^3$. When we speak of avoidance, we always mean in the one line notation for the canon multipermutation unless otherwise specified. When referring to a canon permutation, $\pi$, we will often speak of its corresponding permutation $\sigma_\pi$, as well as the entries in its Young tableau $T_\pi$. We use $\pi_{i,j}$ to denote the entry in the $i$-th row and $j$-th column of $T_\pi$. In $\pi$, this corresponds to the position of the $j$-th occurrence of $(\sigma_\pi)_i$.

 \begin{remark}
     We note that this framework can also be phrased in terms of posets. Since the Young diagram poset of shape $(k^n)$ is the product $[k] \times [n]$ of two chains, a canon permutation is a linear extension of $[k] \times [n]$ together with a choice of $\sigma_\pi \in \S_n$. This is the viewpoint developed in \cite{BD24}. One note of warning, their multiplicity parameter $m$ is our $k$. This poset perspective motivates much of Section \ref{sec: symmetry}.
 \end{remark}

 We denote the decreasing permutation of length $n$ by $\delta_n = n(n-1)\cdots21$ for ease of notation. Also for disjoint sets $S_1,S_2 \subset [kn]$ we say that $S_1 < S_2$ if $\max(S_1) < \min(S_2)$, meaning every element of $S_1$ is less than every element of $S_2$.

 Canon permutations are also invariant under certain standard symmetries. Given $\pi = \pi_1 \cdots \pi_N$, we let $\pi^r = \pi_N \pi_{N-1} \cdots \pi_2 \pi_1$ denote the \emph{reverse} of $\pi$. We denote by $\pi^c$, the permutation where $i$ is replaced by $n-i+1$, this is called the \emph{complement} of $\pi$; for example if $\pi = 212312133$ then $\pi^c = 232132311$. We note that both of these symmetries preserve the $k$-regularity of $\pi$; if the word was not regular, the usual definition of complement would not work.
 
 \begin{proposition} \label{prop: reverse-complement}
     If $\pi$ is a canon permutation, then so are $\pi^r$ and $\pi^c$.
 \end{proposition} 
 
 \begin{proof}
     If $\pi = (\sigma_\pi, T_\pi)$, then $\pi^r = (\sigma_\pi^r, T_\pi^e)$ where $T^e_{i,j} = nk - T_{n-i+1,k-j+1}+1$ is the evacuation, sometimes called the Sch\"utzenberger involution, of $T$. Similarly, $\pi^c = (\sigma_\pi^c, T_\pi)$. 
 \end{proof}
 
 We also note that as with standard permutations, these operations commute with one another; that is, $\pi^{rc} = \pi^{cr}$. We now consider two simple results that give the reader a sense for how the arguments will go in the rest of this paper. We first consider an extreme example where $\sigma_\pi$ is fixed and $T_\pi$ is unrestricted,

 \begin{proposition} \label{prop: standardTab}
    For $n \geq 1$, $k \geq 1$, let $\tau_{j,\ell} = 1^j2^\ell$ for any $1 \leq j,\ell$ with $j+\ell = k+1$. Then $c_n^{k}(\tau_{j,\ell})$ is the number of standard Young tableaux of shape $(k^n)$:
    \[
    C_n^{(k)} = (nk)! \;\frac{\Phi(n) \cdot \Phi(k)}{\Phi(n+k)}
    \]
    where $\Phi(m) = \prod_{i=0}^{m-1} i!$, the superfactorial function.
\end{proposition}

\begin{proof}
    First, given $\pi \in \C_n^{k}(\tau_{j,\ell})$, notice that $\sigma_\pi = \delta_n$. If there is some increase, $(\sigma_\pi)_i < (\sigma_\pi)_{i+1}$, then since $T_\pi$ is standard, $\pi_{i,1} < \pi_{i,2} < \cdots < \pi_{i,j} < \pi_{i+1,j} < \cdots < \pi_{i+1,k}$. But this means $\pi$ contains $(\sigma_\pi)_i^{j} (\sigma_\pi)_{i+1}^{k-j+1} = (\sigma_\pi)_i^j (\sigma_\pi)_{i+1}^{\ell}$ as a subsequence, which is a $1^j2^{\ell}$ pattern.

    But now, if $\sigma_\pi = \delta_n$, $T_\pi$ can be any standard Young tableau. Indeed, the only way to get a $1^j2^\ell$ pattern is to have $\pi_{i,j} < \pi_{i-m,j}$ for some $i,m$ because $T_\pi$ has $k$ columns and $j+\ell = k+1$, but this is impossible because $T_\pi$ is standard.
\end{proof}

\begin{remark}
    This also holds for $\tau_{j,\ell}^r$, $\tau_{j,\ell}^{rc}$ and $\tau_{j,\ell}^c$ by Proposition \ref{prop: reverse-complement}. One important note is that the hypothesis $j+\ell = k+1$ cannot be weakened. If $j+\ell > k+1$, we can no longer conclude that $\sigma_\pi = \delta_n$. For example, when $k = 3$ and $n = 2$, the canon permutation $\pi = 121212$ has $\sigma_\pi = 12$, and still avoids $1^22^3$; here $j+\ell = 5 > 4 = k+1$.
\end{remark}

We can recover a known result about nonnesting permutations avoiding $112$,

\begin{corollary}
    For $n \geq 1$, $c_n^2(112) = C_n$, the $n$-th Catalan number.
\end{corollary}

\begin{proof}
    By Proposition \ref{prop: standardTab}, this is the case where $k = 2$ and we have $\tau_{2,1}$, so $c_n^2(112)$ is the number of standard tableaux of shape $(2^n)$ which is $C_n$.
\end{proof}

We will also recover this result as a special case of Theorem \ref{thm: 112}, where we classify $c_n^k(112)$ for any values of $n$ and $k$. We now consider the opposite extreme, $T_\pi$ will be fixed and $\sigma_\pi$ will be unrestricted. This is a generalization of \cite[Theorem 2.2]{EL25}, the idea behind the proof is essentially identical, but we rephrase it in terms of Young tableaux.

\begin{proposition} \label{prop: 212}
    For $n \geq 1$, $c_n^k(\tau) = n!$ for any $k$ and $\tau \in \{121,212\}$.
\end{proposition}

\begin{proof}
    Up to Proposition \ref{prop: reverse-complement}, it suffices to prove this result for $\tau = 212$. For a canon permutation to avoid $212$ we must have $\pi_{i,k} < \pi_{i+1,1}$. Otherwise, $\pi_{i,1} < \pi_{i+1,1} < \pi_{i,k}$ and $\pi_{i+1,1} < \pi_{i,k} < \pi_{i+1,k}$. If $(\sigma_\pi)_i < (\sigma_\pi)_{i+1}$ then the second string of inequalities implies that $(\sigma_\pi)_{i+1} (\sigma_\pi)_i (\sigma_\pi)_{i+1}$ is a subsequence of $\pi$. If instead $(\sigma_\pi)_{i} > (\sigma_\pi)_{i+1}$ the first string of inequalities implies that $(\sigma_\pi)_i (\sigma_\pi)_{i+1} (\sigma_\pi)_i$ is a subsequence of $\pi$. In either case $\pi$ contains a $212$.

    This shows that $T_\pi$ must have all entries in row $i$ less than all entries in row $i+1$ for $1 \leq i \leq n-1$. Any permutation paired with this tableau will clearly result in a canon permutation avoiding $212$. This completes the proof.    
\end{proof}

This means the Stirling canon permutations are generally not interesting because studying avoidance in $\pi$ that are both canon and Stirling is equivalent to just considering avoidance in $\sigma_\pi$. We do note, though, that not all Stirling permutations are canon and vice versa. Indeed $\pi = 1212$ is canon, but not Stirling, and $1221$ is Stirling but not canon. This distinction is important, because although being canon is a natural generalization of nonnesting permutations, which were a generalization of $212$ avoiding permutations, canon permutations are not naturally defined by an avoidance condition. In light of Proposition \ref{prop: 212}, we now focus our effort on the non-Stirling canon permutations.
 
\section{Non-Stirling Canon Permutations}

In this section we enumerate $\C_n^k(\sigma,\tau)$ for $\sigma \in \{112,122,211,221\}$ and $\tau \in \S_3\cup \{\varnothing\}$. Up to Proposition \ref{prop: reverse-complement}, it suffices to enumerate $\C_n^k(112,\tau)$ for $\tau \in \S_3\cup \{\varnothing\}$, this is the case we focus on. If we consider $112$ avoiding canon permutations, the enumeration differs with $k$, and we recover the $k$-Catalan numbers,

\begin{theorem} \label{thm: 112}
For $n \geq 1$ and $k \geq 2$, we have
\[
c_n^k(\tau) = \frac{1}{n} \binom{kn}{n-1} = \C_{k,n}
\]
for $\tau \in \{112,122,211,221\}$.
\end{theorem}

\begin{proof}
    We prove this for $\tau = 112$, the other options for $\tau$ follow by taking reverses and complements. First, in order to avoid $112$ the permutation associated to the Young tableau must be the decreasing permutation. If not, there is an increase in the corresponding permutation $(\sigma_\pi)_i < (\sigma_\pi)_{i+1}$. Consider the entries in rows $i$ and $i+1$ of $T_\pi$. We have $\pi_{i,1} < \pi_{i,2} < \pi_{i+1,2}$ since the Young tableau is standard, so $\pi$ has subsequence $(\sigma_\pi)_i (\sigma_\pi)_i (\sigma_\pi)_{i+1}$, which forms a $112$ since $(\sigma_\pi)_i < (\sigma_\pi)_{i+1}$. This means $\sigma_\pi = \delta_n$, the decreasing permutation.

    Next, we claim that $\pi_{i,k} < \pi_{i+1,2}$ for $1 \leq i \leq n-1$. If not, there is some $i$ for which $\pi_{i,k} > \pi_{i+1,2}$, but then we have $\pi_{i+1,1} < \pi_{i+1,2} < \pi_{i,k}$. This implies the existence of the subsequence $(\delta_n)_{i+1}(\delta_n)_{i+1} (\delta_n)_i$ in $\pi$, which is a $112$ pattern.

    Finally, we claim that for any standard Young tableau $T$ satisfying $T_{i,k} < T_{i+1,2}$ for $1 \leq i \leq n-1$, $\pi = (\delta_n,T)$ is a $112$ avoiding canon permutation. It is certainly a canon permutation, it remains to show that it avoids $112$. In order for $\pi$ to contain a $112$ the larger entry must appear in a smaller row of $T$, so the positions of the $112$ subsequence would have to be $\pi_{i,1} < \pi_{i,\ell} < \pi_{j,k}$ for some $\ell \le k$ and $j < i$. But this never occurs because $\pi_{j,k} < \pi_{i,2}$ if $j < i$.

    We can therefore construct any permutation $\pi \in \C_n^k(112)$ by deciding whether to place entries in the first column, or in columns $2,\dots,k$ since the entries in those columns must appear in the order $\pi_{1,2} < \pi_{1,3} < \cdots < \pi_{1,k} < \pi_{2,2} < \cdots < \pi_{n,k}$. We must still respect the standard condition on the tableau which is equivalent to: if we have placed $i$ entries in the first column, we can place at most $(k-1)i$ entries in columns $2,\dots,k$, otherwise we would be placing an entry in column $2$ of a row that does not have an entry in column $1$.

    We can think of this construction process via a bijection with lattice paths. The previous paragraphs show that $\pi \in \C_n^k(112)$ are in bijection with lattice paths of $n$ East steps and $(k-1)n$ North steps from $(0,0)$ to $(n,(k-1)n)$ that lie weakly below the line $y = (k-1)x$. Indeed, each east step corresponds to placing an entry in the first column, and north steps correspond to placing an entry in columns $2,\dots,k$ of our tableau. We can decompose such lattice paths as
    \[
    EA_1N A_2 N \cdots A_{k-1}NA_k
    \]
    where each $A_i$ is another such lattice path of length $n_i$ where $n_i$ could be zero, and $n_1 + \cdots + n_{k} = n-1$. This shows that $F(x) = \sum_{n \geq 0} c_n^k(112) x^n$ satisfies
    \[
    F(x) = xF(x)^{k} + 1
    \]
    which by Lagrange inversion is precisely the generating function for the stated binomial.
\end{proof}   

\begin{remark}
    Notice when $k = 2$, this recovers \cite[Theorem 2.1]{EL25} where the authors show the number of nonnesting permutations avoiding $112$ of length $n$ is $C_n$, the $n$-th Catalan number. The $k = 2$ case is special because once you show that $\sigma_\pi = \delta_n$, $T_\pi$ can be any standard Young tableau of shape $2^n$, of which there are $C_n$.
\end{remark}

\begin{example}
     Let's consider an example of this bijection between lattice paths and $\pi \in \C_n^k(112)$. For example, given $\pi = 323132211$, represented as a tableau this is,
    \begin{center}
\ytableausetup{centertableaux}

   \begin{ytableau}
       \none[3] &1 &3 &5\\
       \none[2] &2&6&7\\
       \none[1] &4&8&9
   \end{ytableau}
   \end{center}

   This corresponds to the lattice path $EENENNNNN$. Since this is a standard tableau, the number of $N$ steps can never exceed $(k-1)$ times the number of $E$ steps because this would imply you place an entry in column $2$ of a row that does not have an entry in column $1$. Similarly, given such a lattice path, we can easily reverse this process to find an allowable filling of our Young tableau producing a permutation in $\C_n^k(112)$.
\end{example}

\begin{corollary} \label{cor: 112-123}
    For $n \geq 1$ and $k \geq 2$,
    \[
    c_n^k(112,123) = \frac{1}{n} \binom{nk}{n-1}
    \]
\end{corollary}

\begin{proof}
    All of the permutations produced in the proof of Theorem \ref{thm: 112} necessarily avoid $123$. Indeed, given $\pi \in \C_n^k(112)$, suppose for sake of contradiction that $\pi$ contains a $123$. Since $\sigma_\pi = \delta_n$, we must have $\pi_{i_3,j_3} < \pi_{i_2,j_2} < \pi_{i_1,j_1}$ for $i_1 < i_2 < i_3$. But since $T_\pi$ is standard, this means the $j_i$ are pairwise distinct. In particular, at most one of the $j_i$ can occur in the first column. Suppose the two columns that are not column $1$ are $1 < j_a,j_b$ where $a < b$. Since $\pi \in \C_n^k(112)$ and $i_a < i_b$, by the proof of Theorem \ref{thm: 112}, we must have $\pi_{i_a,j_a} < \pi_{i_b,j_b}$. This is a contradiction.
\end{proof}

\begin{theorem} \label{thm: 112-132}
    For $n \geq 1$ and $k \geq 2$, we have
    \[
    c_n^k(112,\tau) = k^{n-1}
    \]
    for $\tau \in \{132,213\}$.
\end{theorem}

\begin{proof}
    
    Since $\pi \in \mathcal{C}_n^k(112)$, we showed in Theorem \ref{thm: 112} that $\sigma_\pi = \delta_n$. 
    
    Now we claim that for $n \ge 1$ and $k \ge 2$, for $\pi \in \C_n^k(112)$, $\pi$ avoids $\tau$ if and only if $\pi_{i-2,k} < \pi_{i,1}$ for $3 \le i \le n$. If not, suppose $\pi_{i-2,k} > \pi_{i,1}$, then $\pi_{i,1} < \pi_{i-2,k} < \pi_{i-1,k}$ which means $\pi$ contains $(\delta_n)_i (\delta_n)_{i-2}(\delta_n)_{i-1}$ which is a $132$ pattern. We also have $\pi_{i-1,1} < \pi_{i,1} < \pi_{i-2,k}$, which means $\pi$ contains $(\delta_n)_{i-1} (\delta_n)_{i} (\delta_n)_{i-2}$ which is a $213$ pattern. 
    
    Conversely, suppose $\pi_{i-2,k} < \pi_{i,1}$ for $3 \le i \le n$, and for sake of contradiction assume we have an occurrence of $\tau$ in $\pi$. Say the occurrence corresponds to positions $\pi_{r_1,c_1} < \pi_{r_2,c_2} < \pi_{r_3,c_3}$ in $\pi$. Since $\sigma_\pi = \delta_n$, if $\tau = 132$ then $r_2<r_3<r_1$ and if $\tau = 213$ we have $r_3< r_1 < r_2$. In the first case, $\pi_{r_2,c_2} > \pi_{r_1,c_1} \ge \pi_{r_1,1} \ge \pi_{r_2+2,1}$ since $r_2 < r_3 < r_1$. In the second case, $\pi_{r_3,c_3} > \pi_{r_2,c_2} \ge \pi_{r_2,1} \ge \pi_{r_3+2,1}$ by the same reasoning. In either case, we get a contradiction. 

    Now, in combination with the requirements on the tableau for $\pi$ to avoid $112$, this implies we have
    \[
    \pi_{1,1} < \{\pi_{2,1}, \pi_{1,2},\dots,\pi_{1,k}\} < \{\pi_{3,1},\pi_{2,2},\dots,\pi_{2,k}\}< \cdots < \{\pi_{n,1},\pi_{n-1,2},\dots,\pi_{n-1,k}\} < \{\pi_{n,2},\dots,\pi_{n,k}\}.
    \]
    In terms of the tableau, this means the colored squares consist of consecutive integers that increase from top to bottom:
     \begin{center}
\ytableausetup{centertableaux}

   \begin{ytableau}
       1 &{\color{red} \ast} &\cdots &{\color{red} \ast}\\
       {\color{red} \ast} &{\color{blue} \ast} &\cdots&{\color{blue} \ast}\\
       {\color{blue} \ast} &&\cdots&nk
   \end{ytableau}
   \end{center}
    There are $k$ ways to arrange the entries in the positions corresponding to all but the first and last subset of positions. Indeed, they correspond to placing consecutive entries in the skew tableau

   \begin{center}
\ytableausetup{centertableaux}
   \begin{ytableau}
       \none & &\cdots & \\
       {}
   \end{ytableau}
   \end{center}
   where the top row has $k-1$ boxes. There is only one way to arrange the entries in positions $\pi_{n,2},\dots,\pi_{n,k}$ since we must have $\pi_{n,2} < \cdots < \pi_{n,k}$. Any such tableau paired with the decreasing permutation gives a $112$ and $\tau$ avoiding canon permutation. This completes the proof.
\end{proof}

\begin{example}
    Consider $\pi = 434433212211 \in \C_4^3(112,132)$. This has
    \[
    (\sigma_\pi,T_\pi) = 
    \begin{ytableau}
        \none[4] &1&{\color{red}3}&{\color{red} 4}\\
        \none[3] &{\color{red} 2}&5&6\\
        \none[2] &7&9&10\\
        \none[1] &8&11&12
    \end{ytableau}
    \]
    Notice if we were to rearrange the red numbers in any allowable way, it would still result in a $112$ and $132$ avoiding canon permutation. For example, if we swap $2$ and $3$,
    \[
    \begin{ytableau}
        \none[4] &1&{\color{red}2}&{\color{red} 4}\\
        \none[3] &{\color{red} 3}&5&6\\
        \none[2] &7&9&10\\
        \none[1] &8&11&12
    \end{ytableau}
    \]
    corresponds to the permutation $443433212211 \in \C_4^3(112,132)$.
\end{example}

Notice that Theorem \ref{thm: 112-132} also implies that avoiding $112$ and $132$ is equivalent to avoiding $112$ and $213$ from which we deduce the following corollary,

\begin{corollary}
    For $n \ge 1$ and $k \ge 2$, we have
    \[
    c_{n}^k(112,132,213) = k^{n-1}
    \]
\end{corollary}

\begin{proof}
    The proof of Theorem \ref{thm: 112-132} shows that the canon permutations we produced to avoid $112$ and $132$ necessarily also avoid $213$. This is largely the content of the first two paragraphs of the proof.
\end{proof}

Interestingly, the canon permutations avoiding $\{112,312\}$ and those avoiding $\{112,231\}$ are equinumerous, but unlike the previous case they correspond to distinct tableaux.

\begin{theorem} \label{thm: 112-312}
    For $n \geq 2$, $k \geq 2$ and $\tau \in \{312,231\}$, we have
    \[
    c_n^k(112,\tau) = k.
    \]
\end{theorem}

\begin{proof}
    Since $\pi \in \mathcal{C}_n^k(112)$, we know $\sigma_\pi = \delta_n$. So it remains to characterize the corresponding tableau. To avoid $312$, we must have $\pi_{i,1} > \pi_{i-1,k}$ for $3 \leq i \leq n$. Stated another way, all entries in row $i$ must be larger than all entries in row $i-1$. If not, then $\pi_{i-2,1} < \pi_{i,1} < \pi_{i-1,k}$, which means $\pi$ contains the subsequence $(\delta_n)_{i-2}(\delta_n)_i (\delta_n)_{i-1}$ which is a $312$ pattern. Similarly, to avoid $231$ we must have $\pi_{i,1} > \pi_{i-1,k}$ but now for $2 \le i \le n-1$. If not, then $\pi_{i,1} < \pi_{i-1,k} < \pi_{i+1,k}$ which means $\pi$ contains the subsequence $(\delta_n)_{i} (\delta_n)_{i-1} (\delta_n)_{i+1}$ which is a $231$ pattern.

    First consider the case where $\tau = 312$. Since $\pi_{1,1} = 1$, the above in combination with the fact that $\pi_{1,2} < \pi_{1,3} < \cdots < \pi_{1,k} < \pi_{2,2} < \cdots < \pi_{2,k}$ from Theorem \ref{thm: 112} means once we select $\pi_{2,1}$ from the entries $\{2,\dots,k+1\}$, this completely determines $\pi$. Since there are $k$ choices for $\pi_{2,1}$ this completes the proof.

    If instead $\tau = 231$, a similar argument applies with the roles of the first and last rows exchanged. The inequality $\pi_{i,1} > \pi_{i-1,k}$ for $2 \le i \le n-1$ says that every entry of row $i$ exceeds every entry of row $i-1$ for $2 \le i \le n-1$. So rows $1,\dots,n-2$ are just the consecutive blocks $\{1,\dots,k\}, \dots, \{(n-3)k+1,\dots,(n-2)k\}$ and are completely determined, and $\pi_{n-1,1} = (n-2)k+1$. By Theorem \ref{thm: 112} the remaining entries satisfy $\pi_{n-1,2} < \pi_{n-1,3} < \cdots < \pi_{n-1,k} < \pi_{n,2} < \cdots < \pi_{n,k}$, so the only freedom we have is choosing $\pi_{n,1}$, which must exceed $\pi_{n-1,1}$ and precede $\pi_{n,2}$. Again there are $k$ choices, each determining $\pi$.

    It is also clear that all of the tableaux produced in this way correspond to $\pi$ that avoid $312$ in the first case or $231$ in the second case.
\end{proof}

\begin{example}
    Let $n = 4$ and $k = 3$. Suppose we pick $\pi_{2,1} = 3$. This uniquely produces
    \[
    T_\pi = \begin{ytableau}
        \none[4] &1&2&4\\
        \none[3] &{\bf 3}&5&6\\
        \none[2] &7&8&9\\
        \none[1] &10&11&12
    \end{ytableau}
    \]
    so that $\pi = 443433222111 \in \C_4^3(112,312)$.
\end{example}

\begin{remark} \label{rem: 321}
    For completion's sake, we note that $c_n^k(112,321) = 0$ for $n \geq 3$ because $\sigma_\pi = \delta_n$ necessarily, so as soon as we have at least $3$ rows, we must have a $321$ pattern. We also note that neither case from Theorems \ref{thm: 112-132} or \ref{thm: 112-312} follows from the other by Proposition \ref{prop: reverse-complement}. Indeed, the orbit of $\{112,132\}$ under reverse and complement is $\{211,231\}$, $\{221,312\}$ and $\{122,213\}$, so $\{112,213\}$ and $\{112,231\}$ are genuinely separate computations. Similarly, the orbit of $\{112,312\}$ under reverse and complement is $\{211,213\},\{221,132\},\{122,231\}$ which does not contain $\{112,231\}$.
\end{remark}

This motivates the following definition,

\begin{definition} \label{def: canon-wilf}
    Let $k \ge 2$. We say that two sets of patterns $\Lambda_1, \Lambda_2 \subseteq \mathcal{P}$ are \emph{$k$-canon-Wilf equivalent} if $c_n^k(\Lambda_1) = c_n^k(\Lambda_2)$ for all $n \ge 1$. If $\Lambda_1$ and $\Lambda_2$ are $k$-canon-Wilf equivalent for every $k \ge 2$, we say they are \emph{canon-Wilf equivalent}.
\end{definition}

Analogous notions of Wilf equivalence have been studied for patterns in rooted labeled forests \cite{R24}, for classical and consecutive patterns in rooted forests \cite{GP} and for arrow patterns in \cite{AL26}. With this definition, we can restate some of the above results,

\begin{corollary}
    The following sets of patterns are canon-Wilf equivalent:
    \begin{itemize}
        \item $\{112,132\}$, $\{112,213\}$ and $\{112,132,213\}$,
        \item $\{112,312\}$, $\{112,231\}$ and $\{112,132,231\}$,
        \item $\Lambda$, $\Lambda^r$, $\Lambda^c$ and $\Lambda^{rc}$, for any set of patterns $\Lambda$.
    \end{itemize}
\end{corollary}

\begin{proof}
    For the third set in the second item, note that the $231$ condition of Theorem \ref{thm: 112-312} applied at $i-1$ gives $\pi_{i-2,k} < \pi_{i-1,1} \le \pi_{i,1}$, so $\C_n^k(112,231) \subseteq \C_n^k(132)$. This means adding $132$ to the avoidance set changes nothing. All of the other results are immediate corollaries of Theorems \ref{thm: 112-132}, \ref{thm: 112-312} and Proposition \ref{prop: reverse-complement}.
\end{proof}

\begin{remark}
    By contrast $\{112,312,231\}$ is not equivalent to any of the above. Indeed, the two conditions together force every row of $T_\pi$ to lie entirely before the next, so the tableau is unique and $c_n^k(112,312,231) = 1$ for $n \ge 3$.
\end{remark}

We again note that this handles all the pairs $c_n^k(\sigma,\tau)$ where $\sigma \in \{112,122,211,221\}$ and $\tau \in \S_3 \cup \{\varnothing\}$ by taking reverses and complements. For example, since $c_n^k(112,132) = k^{n-1}$ this is also true of $c_n^k(221,312)$ by taking the complement, and $c_n^k(211,231)$ by taking the reverse, and $c_n^k(122,213)$ by taking the reverse complement.

\section{Connection to $k$-ary Trees} \label{sec: trees}

A $k$-ary tree is a rooted ordered tree in which each vertex has $k$ or $0$ children. In Theorem \ref{thm: 112} we showed that $c_n^k(112)$ has the same enumeration as the number of $k$-ary trees with $n$ internal vertices. In this section, we explore this connection a little further. First, we present an explicit bijection between $\mathcal{C}_n^k(112)$ and $k$-ary trees on $n$ internal vertices. Then, we explore how avoidance in the canon permutations interacts with avoidance in the $k$-ary trees as defined in \cite{GPPT11}.

Given a permutation $\pi \in \C_n^k(112)$, form a $k$-ary tree via the following procedure. First create a root vertex with label $0$. Let $R = 0$, this will record the label of the most recent vertex, and $C = 1$, this will record the label of the vertex we are placing.
\begin{itemize}
    \item If $C$ occurs in the first column of the tableau $T_\pi$, create a child of $R$ and label it $C$. Set $R = C$ and $C = C+1$ and iterate.
    \item If $C$ does not occur in the first column of $T_\pi$ and the parent of $R$ has less than $k$ children, create another child of the parent of $R$ to the right of all the current children, label it $C$. Set $R = C$ and $C = C+1$ and iterate.
    \item If $C$ does not occur in the first column of $T_\pi$ and the parent of $R$ has $k$ children, set $R$ equal to the parent of the current $R$, keep $C$ the same, and iterate.
\end{itemize} 

For example,    

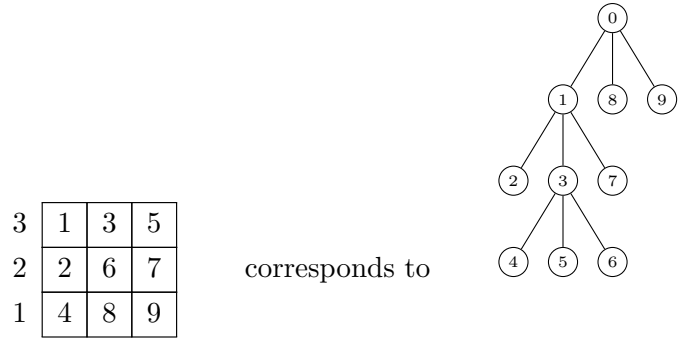
\begin{center}
\ytableausetup{centertableaux}

   \begin{ytableau}
       \none[3] &1 &3 &5\\
       \none[2] &2&6&7\\
       \none[1] &4&8&9
   \end{ytableau}
\qquad corresponds to \qquad
\begin{forest}
    for tree={circle, draw, minimum size = 1em, inner sep = 1pt, font = \tiny}
    [0,circle,draw
    [1
        [2]
        [3
            [4] [5] [6]]      
        [7]]
    [8]
    [9]]
\end{forest}
\end{center}

The tree itself is not labeled. The labels record which entry of the tableau creates each vertex. The procedure is well defined. A straightforward induction shows that at every stage the vertices with fewer than $k$ children are exactly the ancestors of $R$, since each new vertex is added as the rightmost child of an ancestor of the previous one. Moreover, suppose $C$ does not occur in the first column and let $m = |\{i \; | \; \pi_{i,1} \le C\}|$, so that the tree currently has $m$ internal vertices, hence $km$ available slots, and $C-1$ non-root vertices. The lattice path condition of Theorem \ref{thm: 112} says that among the first $C$ entries, the number lying outside the first column is at most $(k-1)m$. This means $C \le km$, so $C - 1 < km$ and some ancestor of $R$ has a free slot. To invert the map, label the non-root vertices $1,\dots,nk$ in depth first order. The first column of $T_\pi$ consists of exactly those labels that are leftmost children of their parent.

Notice that $\{\pi_{1,1}-1, \pi_{2,1}-1,\dots,\pi_{n,1}-1\}$ are precisely the internal vertices of the corresponding tree, so this construction does produce trees with $n$ internal vertices.

We can recover some results of Gabriel, Peske, Pudwell and Tay \cite{GPPT11} from this perspective. They study avoidance in $k$-ary trees, with a specific focus on ternary trees. We can reframe their avoidance questions in our context. For them, an ordered rooted tree $T$ contains $t$ as a tree pattern if $t$ is a contiguous, rooted, and ordered subtree of $T$. $T$ avoids $t$ if there is no such subtree of $T$.

They define specific families of subtrees and study avoidance of those families. One such family is $t_{{(2k-1)}_1}$, which looks like

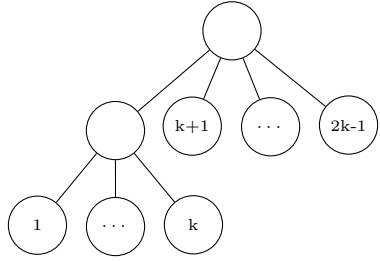
\[
\begin{forest}
    for tree={circle, draw, minimum size = 2em, inner sep = 1pt, font = \tiny}
    [,circle,draw
    [
        [1]
        [$\cdots$]
        [k]]
    [k+1]
    [$\cdots$]
    [2k-1]]
\end{forest}
\]

where this is the $k$-ary tree with two internal vertices. Since the leftmost child of the root is the internal one, it has $2k-1$ leaves. When $k = 3$ it is the tree denoted $t_{5_1}$ in \cite{GPPT11}. We can use canon permutations to generalize their result:

\begin{theorem}
    For $n \ge1$ and $k \ge 2$, the number of $k$-ary trees on $n$ internal vertices avoiding $t_{(2k-1)_1}$ is $\C_{k-1,n}$, the number of $(k-1)$-ary trees on $n$ internal vertices.
\end{theorem}

\begin{proof}
    Using our bijection, we know any $k$-ary tree on $n$ internal vertices corresponds to some $\pi \in \C_n^k(112)$. Since every internal vertex of a $k$-ary tree has exactly $k$ children, a tree contains $t_{(2k-1)_1}$ if and only if it has an internal vertex whose leftmost child is also internal. In the construction above, the vertex labeled $C$ receives a child exactly when $C+1$ occurs in the first column of $T_\pi$, and in that case $C+1$ is its leftmost child. So $C$ is internal if and only if $C+1$ lies in the first column. Hence the tree contains $t_{(2k-1)_1}$ if and only if two consecutive integers $a$, $a+1$ both lie in the first column of $T_\pi$, that is, if and only if the first column has a consecutive ascent. This is the same as the corresponding lattice path never containing a consecutive $EE$.

    Such permutations are in bijection with lattice paths with $n$ east steps and $n(k-2)$ north steps that lie weakly below the line $y = (k-2)x$. To see this, in the lattice path corresponding to $\pi$, which has $n$ east steps and $n(k-1)$ north steps and lies weakly below the line $y = (k-1)x$, remove the $N$ step that must occur directly after each $E$ step. This removes $n$ north steps and produces another valid lattice path that will now lie weakly below the line $y = (k-2)x$ instead of $y = (k-1)x$. This construction is clearly reversible, we could add an additional $N$ step after each $E$ step. We conclude that the number of $k$-ary trees on $n$ internal vertices avoiding $t_{(2k-1)_1}$ is the number of $(k-1)$-ary trees on $n$ internal vertices.
\end{proof}

When $k = 3$, this recovers a similar result of Gabriel, Peske, Pudwell and Tay (\cite[Section 2.1]{GPPT11}). Their result is stated in terms of the number of ternary trees with $m$ leaves that avoid $t_{5_1}$, our equivalent restatement considers the number of internal vertices instead. This avoids the sequence being interspersed with $0$'s because there are no ternary trees with an even number of leaves:

\begin{corollary}
    The number of $3$-ary trees on $n$ internal vertices avoiding $t_{5_1}$ is $\frac{1}{n} \binom{2n}{n-1} = C_n$, the $n$-th Catalan number.
\end{corollary}

To see why these results are equivalent, notice if a $k$-ary rooted ordered tree has $j$ internal vertices, it has $kj + 1$ total vertices since each internal vertex has $k$ children and we add in the root. A tree with $j$ internal vertices therefore has $kj + 1 - j = (k-1)j + 1$ total leaves. When $k = 3$, our result corresponds to the number of $3$-ary trees on $2n+1$ leaves being equal to $\frac{1}{n} \binom{2n}{n-1} = C_n$, which precisely recovers the odd entries from their result.

It would be interesting to further explore this connection and see how avoidance in $k$-ary trees interacts with avoidance in canon permutations.

\section{Avoiding Symmetric Sets} \label{sec: symmetry}

Let $\Lambda$ be a set of patterns on the alphabet $[m]$. For $\sigma \in \S_m$ and a word $\lambda$ on $[m]$, write $\sigma\lambda$ for the word obtained from $\lambda$ by relabeling each letter $i$ as $\sigma(i)$. We say that $\Lambda$ is \textit{symmetric} if
\[
\sigma \cdot \Lambda := \{\sigma\lambda \; | \; \lambda \in \Lambda\} = \Lambda \qquad \text{for every } \sigma \in \S_m,
\]
equivalently, $\Lambda$ is a union of $\S_m$-orbits of words. For example, both the nonnesting and noncrossing avoidance sets $\{1221,2112\}$ and $\{1212,2121\}$ are symmetric. Given a pattern $\lambda$ on the alphabet $[m]$, we write
\[
\O(\lambda) := \S_m \cdot \lambda = \{\sigma\lambda \; | \; \sigma \in \S_m\}
\]
for the $\S_m$-orbit of $\lambda$.  Every $\O(\lambda)$ is symmetric, and a set of patterns is symmetric precisely when it is a union of orbits $\O(\lambda)$.  When $m = 2$ the orbit is just $\O(\lambda) = \{\lambda,\lambda^c\}$. For instance, $\O(1221) = \{1221,2112\}$ and $\O(12^j1) = \{12^j1,21^j2\}$. We consider the sets $\O(1221)$, $\O(1212)$, $\O(12^j1)$, $\O((12)^j)$ and $\O(1^a21^b)$ below.

\begin{remark} \label{rem: orbits}
    Since the complement $i \mapsto m-i+1$ is itself an element of $\S_m$, we have $\O(\lambda)^c = \O(\lambda)$, so every symmetric set is automatically closed under complements.  Reversal, on the other hand, gives $\O(\lambda)^r = \O(\lambda^r)$, so by Proposition \ref{prop: reverse-complement} the sets $\O(\lambda)$ and $\O(\lambda^r)$ are canon-Wilf equivalent.
\end{remark}

Symmetric avoidance in canon permutations is related to avoidance in $k$-regular lattice words. Recall that a \emph{lattice word} $\pi$ on $[n]$ is a word with the property that in every prefix of the word, the number of occurrences of the letter $i$ is greater than or equal to the number of occurrences of the letter $i+1$ for all $1 \leq i \leq n-1$. We let $\mathcal{L}_n^k$ denote the set of $k$-regular lattice words of length $kn$ on the alphabet $[n]$.

\begin{theorem} \label{thm: symmetric}
    If $\Lambda$ is symmetric, then $c_n^k(\Lambda) = n! |\mathcal{L}_{n}^k(\Lambda)|$.
\end{theorem}

\begin{proof}
    If $\Lambda$ is symmetric, this means that $\pi$ avoids $\Lambda$ if and only if any permutation of $\sigma_\pi$, the underlying letters of the permutation, also avoids $\Lambda$. We can therefore assume $\sigma_\pi = 12\cdots n$. For each distinct canon permutation in $\C_{n}^{k,{\rm id}_n}(\Lambda)$ we generate $n!$ distinct permutations in $\C_n^k(\Lambda)$ by permuting $\sigma$. Indeed, in order for two canon permutations to be equal, their underlying permutation and their tableau must be the same. By choosing distinct permutations in $\C_{n}^{k,{\rm id}_n}(\Lambda)$, we necessarily choose permutations with distinct tableaux, so changing the underlying permutation cannot make them equal.

    To see the connection to lattice words, note that a Young tableau $T$ of shape $(k^n)$ corresponds to a lattice word $\pi_T$ defined by $(\pi_T)_i = j$ if $i$ appears in row $j$ of the tableau $T$ for $1 \leq i \leq nk$. Since $T$ has shape $(k^n)$, the resulting lattice word is $k$-regular. In the case that $\sigma_\pi = {\rm id}_n$, the canon permutation $\pi = ({\rm id}_n, T_\pi)$ is precisely the lattice word corresponding to $T_\pi$.
\end{proof}

When $n \geq 1$ and $k = 2$, $c_n^2(1221,2112) = c_n^2$, since being canon in this case is equivalent to avoiding $1221$ and $2112$. It is well known that $c_n^2 = n! C_n$. This is also clear from the previous result because the number of standard Young tableaux of shape $(2^n)$ is $C_n$ and all of them result in canon permutations. Avoidance in binary lattice words was investigated in \cite{PS26}; this is the case where $n=2$ and there is no regularity condition. If we impose regularity, these binary lattice words are now in bijection with Dyck paths, so one is counting Dyck paths with certain restrictions on their up and down patterns. When discussing avoidance in a lattice word $\pi$, we will continue to use $\pi_{i,j}$ to denote the position of the $j$-th occurrence of entry $i$.

How many canon permutations are also nonnesting when $k > 2$? As in the $k=2$ case, the answer does not depend on $k$,

\begin{proposition} \label{prop: 1221-2112}
    For $n \ge 1$ and $k \ge 3$,
    \[
    c_n^k(\O(1221)) = n! \, 2^{n-1}.
    \]
\end{proposition}

\begin{proof}
    Notice here that since $\Lambda$ is symmetric, we may assume that $\sigma_\pi = 12\cdots n$. We want to argue that for $k \geq 3$, $|\mathcal{L}_n^k(\O(1221))| = 2^{n-1}$. To avoid $1221$, we must have that for any fixed entry $i$ with $1 \le i \le n-1$ of the lattice word $\pi$, $\{\pi_{i,2},\dots,\pi_{i,k}\} < \pi_{i+1,2}$. Otherwise for some $j \geq 2$, $\pi_{i,1} < \pi_{i+1,1} < \pi_{i+1,2} < \pi_{i,j}$ so that $\pi$ contains $i (i+1)(i+1)i$ which is a $1221$-pattern.

    We also must have $\pi_{i+1,1}$ larger than $\{\pi_{i,1},\dots,\pi_{i,k-1}\}$ and only possibly smaller than $\pi_{i,k}$. If not, say $\pi_{i,j} > \pi_{i+1,1}$ for $j < k$, then $\pi_{i+1,1} < \pi_{i,j} < \pi_{i,j+1} < \pi_{i+1,j+1}$ so that $\pi$ contains $(i+1)ii(i+1)$ which is a $2112$-pattern.

    This means that to satisfy the avoidance conditions and remain a lattice word, the word must look like 
    \[
    1^{k-1} \{12,21\} 2^{k-2} \{23,32\} 3^{k-2} \{34,43\}\cdots(n-1)^{k-2}\{(n-1)n, n(n-1)\} n^{k-1}
    \]
    where the $\{i(i+1),(i+1)i\}$ notation means that we choose one of those options for that position. So for each $2 \leq i \leq n$, $\pi_{i,1} \in \{k(i-1), k(i-1) + 1\}$. Once we choose $\pi_{i,1}$ for $2 \leq i \leq n$, the other entries of $\pi$ are fixed, so in total there are $2^{n-1}$ options.
\end{proof}

 We will now generalize and attempt to understand $c_n^k(\O(12^j1))$ for general $n,k,j$. To get a better grasp on this problem, we notice that $c_n^k(\Lambda)$ is also related to the number of linear extensions of posets refining the Young diagram poset of shape $(k^n)$. The avoidance conditions impose additional relations on the poset, which the following lemma identifies,

\begin{lemma} \label{lem: 12j1-rules}
    Let $n \ge 1$, $k \ge 2$, $2 \le j \le k$, and let $\pi \in \mathcal{L}_n^k$ be a $k$-regular lattice word. Then $\pi$ avoids $12^j1$ and $21^j2$ if and only if for all $2 \le i \le n$, we have
    \[
    \pi_{i,j} > \pi_{i-1,k} \qquad \text{and} \qquad \pi_{i,1} > \pi_{i-1,k-j+1}.
    \]
\end{lemma}

\begin{proof}
    An occurrence of $12^j1$ consists of a letter $a$, then $j$ copies of a letter $b > a$, then a second copy of $a$. Since the underlying permutation is the identity, $a$ occupies some row $i$ and $b$ some row $i' > i$. This means $\pi_{i,1} < \pi_{i',1}$ holds automatically, so a $12^j1$ pattern exists precisely when $\pi_{i',j} < \pi_{i,k}$ for some $i < i'$. As $\pi_{i',j} \ge \pi_{i+1,j}$, it suffices to consider consecutive rows, which gives the first condition.

    An occurrence of $21^j2$ consists of a letter $b$, then $j$ copies of a letter $a < b$, then a second copy of $b$. The $j$ copies of $a$ that follow the first $b$ must be the last $j$ copies of $a$, and a further copy of $b$ is then automatically available because $\pi_{i',k} > \pi_{i,k}$; so such an occurrence exists precisely when $\pi_{i',1} < \pi_{i,k-j+1}$ for some $i < i'$. Again consecutive rows suffice, giving the second condition.
\end{proof}

Equivalently, Lemma \ref{lem: 12j1-rules} states that, in constructing the corresponding Young tableau, for $i \ge 2$ we cannot place any entry in row $i$ until row $i-1$ has at least $k-j+1$ entries, and we cannot place an entry in the final $k-j+1$ positions of row $i$ until row $i-1$ is completely filled. 

\begin{remark} \label{rem: poset-contrast}
    It is worth contrasting Lemma \ref{lem: 12j1-rules} with \cite[Section 4]{BD24}, where subposets of $[k] \times [n]$ are obtained by removing relations of the form $(p,j) < (p,j+1)$. The resulting so-called dissonant canon permutations interpolate between canon permutations and arbitrary multiset permutations. Imposing avoidance conditions moves in the opposite direction, often adjoining additional relations to $[k] \times [n]$ instead of removing them. It is interesting to consider how and when these additional relations still result in a ``nice" poset, we make this more rigorous in Question \ref{Q: palindromic-sym}.
\end{remark}

We are now ready to approach enumerating $c_n^k(\O(12^j1))$. We must break the answer into cases depending on the relationship between $j$ and $k$. We first consider the case where $k > 2(j-1)$:

\begin{theorem} \label{thm: 12j1-2(j-1)<k}
    For $n \geq 1$, $j \geq 2$, and $k > 2(j-1)$,
    \[
    c_n^{k}(\O(12^j1)) = n! \, \binom{2j-2}{j-1}^{n-1}.
    \]
\end{theorem}

\begin{proof}
    By Theorem \ref{thm: symmetric}, it suffices to prove that the number of $k$-regular lattice words avoiding $\O(12^j1)$ is $\binom{2j-2}{j-1}^{n-1}$. We claim in general that such a $k$-regular lattice word must look like
    \[
    1^{k-j+1} w\{1^{j-1},2^{j-1}\} 2^{k-2j+2} w\{2^{j-1},3^{j-1}\} 3^{k-2j+2} \cdots (n-1)^{k-2j+2} w\{(n-1)^{j-1} ,n^{j-1}\} n^{k-j+1} 
    \]
    where each $w\{i^{j-1},(i+1)^{j-1}\}$ is a word on $\{i^{j-1},(i+1)^{j-1}\}$. The key step here is to realize that for any fixed $2 \leq i \leq n-1$ we must have 
    \[
    \{\pi_{i+1,1},\dots,\pi_{i+1,j-1}, \pi_{i,k-j+2},\dots,\pi_{i,k}\} > \{\pi_{i,j},\dots,\pi_{i,k-j+1}\} > \{\pi_{i,1},\dots,\pi_{i,j-1},\pi_{i-1,k-j+2},\dots,\pi_{i-1,k}\},
    \] 
    where here all of these sets are guaranteed to be nonempty because $2(j-1) < k$. It is clear that $\pi_{i,j-1} < \pi_{i,j} < \pi_{i,k-j+2}$ from the usual Young diagram poset structure. The inequalities $\pi_{i,j} > \pi_{i-1,k}$ and $\pi_{i+1,1} > \pi_{i,k-j+1}$ are precisely the two conditions of Lemma \ref{lem: 12j1-rules}. The usual Young diagram poset relations imply the other inequalities. 
    
    Applied iteratively for each row, this shows the corresponding lattice word must have the stated form. Conversely, it is not hard to see that any lattice word of this form satisfies the conditions of Lemma \ref{lem: 12j1-rules}. There are $n-1$ words, and there are $\binom{2j-2}{j-1}$ options for each word. This shows when $2 \leq j$ and $2(j-1) < k$, $|\mathcal{L}_n^k(\O(12^j1))| = \binom{2j-2}{j-1}^{n-1}$, the result then follows from Theorem \ref{thm: symmetric}.
\end{proof}

We note that this recovers Proposition \ref{prop: 1221-2112} since there $j = 2$ and we assumed $k \geq 3$ which is equivalent to $2 < k$. We keep the original theorem as motivation for considering the generalization. We now consider the case where $k = 2(j-1)$. This case has a nice solution in terms of lattice paths that have already appeared in the context of the $(s,t)$-tennis ball problem. To state the result we recall its definition. For integers $s > t \ge 1$ and $m \ge 0$, let $T_m(s,t)$ denote the solution to the $(s,t)$-tennis ball problem, that is, the number of lattice paths from $(0,0)$ to $\big(m(s-t),\, mt\big)$ using unit length north and east steps that lie weakly below the path $(N^tE^{s-t})^m$. We refer the reader to \cite{MN05} for background on the tennis ball problem and its solution.

\begin{theorem} \label{thm: 12j1-k=2(j-1)}
     For $n \geq 1$, $j \ge 2$ and $k= 2(j-1)$,
    \[
    c_n^{k}(\O(12^j1)) = n! \;T_{n-1}(k,j-1).
    \]
\end{theorem}

\begin{proof}

Since $k = 2(j-1)$ we have $k-j+1 = j-1$, so the conditions of Lemma \ref{lem: 12j1-rules} say $\pi_{i,j} > \pi_{i-1,k}$ and $\pi_{i,1} > \pi_{i-1,j-1}$.  Together with the fact that the rows of $T_\pi$ increase, the condition $\pi_{i,1} > \pi_{i-1,j-1}$ forces the cells in columns $1,\dots,j-1$ to increase down the rows, as shown with the black arrows in Figure \ref{fig: posetStructure}, and $\pi_{i,j} > \pi_{i-1,k}$ forces the cells in columns $j,\dots,k$ to increase, as shown with the blue arrows.
    
    \begin{figure}   
 \begin{center}
\begin{tikzpicture}[
    box/.style={rectangle, draw=black, minimum size=0.8cm, outer sep=0pt, anchor=center},
    pathL/.style={->, >=Stealth, black, line width=0.7pt},
    pathR/.style={->, >=Stealth, blue, line width=0.7pt},
    jump/.style={->, >=Stealth, dashed, opacity=0.4}
]

    \foreach \r in {1,...,4} {
        \foreach \c in {1,...,8} {
            \node[box] (c\c r\r) at (\c*0.8, -\r*0.8) {};
        }
    }

    \foreach \r in {1,...,4} {
        \draw[pathL] (c1r\r.center) -- (c2r\r.center);
        \draw[pathL] (c2r\r.center) -- (c3r\r.center);
        \draw[pathL] (c3r\r.center) -- (c4r\r.center);
        
        \ifnum\r<4
            \pgfmathtruncatemacro{\nextR}{\r+1}
            \draw[jump, black] (c4r\r.center) -- (c1r\nextR.center);
        \fi
    }

    \foreach \r in {1,...,4} {
        \draw[pathR] (c5r\r.center) -- (c6r\r.center);
        \draw[pathR] (c6r\r.center) -- (c7r\r.center);
        \draw[pathR] (c7r\r.center) -- (c8r\r.center);
        
        \ifnum\r<4
            \pgfmathtruncatemacro{\nextR}{\r+1}
            \draw[jump, blue] (c8r\r.center) -- (c5r\nextR.center);
        \fi
    }

\end{tikzpicture}
\caption{The two increasing chains of Theorem \ref{thm: 12j1-k=2(j-1)}, drawn for $n = 4$ and $k = 8 = 2(j-1)$ with $j = 5$.  The black chain runs through columns $1,\dots,j-1$ and the blue chain through columns $j,\dots,k$.}   \label{fig: posetStructure}
\end{center}   
    \end{figure}
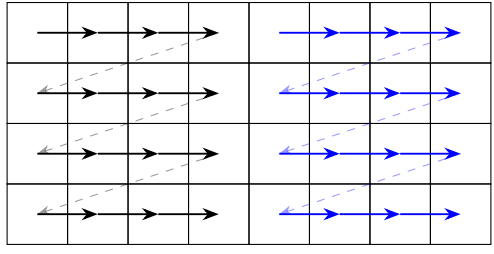

We must also have the standard inequalities imposed on a Young tableau. This forces the first $j-1$ entries in row $1$ to be the numbers $1,\dots,j-1$ and the final $j-1$ entries in row $n$ to be $kn-(j-1)+1,\dots, kn$. We have to then decide to place the remaining entries either in the first or second increasing sequence. However, we must also respect that the first $j-1$ entries in row $i$ have to be smaller than the last $j-1$ entries in the same row. The ways to fill these tableaux are exactly in bijection with lattice paths from $(0,0)$ to $((n-1)(j-1),(n-1)(j-1))$ that are weakly below $(N^{j-1}E^{j-1})^{n-1}$. 

Indeed, given such a lattice path, each $N$ step corresponds to placing an entry in the second increasing path, each $E$ step corresponds to placing an entry in the first increasing path after forcing the first $j-1$ entries to be $1,\dots,j-1$. The restriction on the paths ensures that for every $i$ the $\big(i(j-1)+1\big)$-st
$N$ step comes after the $i(j-1)$-th $E$ step, which means the first $j-1$ entries in each row are smaller than the last $j-1$ entries in that row. It is clear that this construction is a bijection.

The number of such paths is precisely the solution to the $(k,j-1)$ tennis ball problem: the number of lattice paths from $(0,0)$ to $((n-1)(k-(j-1)),(n-1)(j-1)) = ((n-1)(j-1),(n-1)(j-1))$ that lie weakly below $(N^{j-1}E^{k-(j-1)})^{n-1}$.
\end{proof}

\begin{example}
    For example, if $k = 8$ and $j =5$, the following lattice path and tableau filling are in bijection:
    
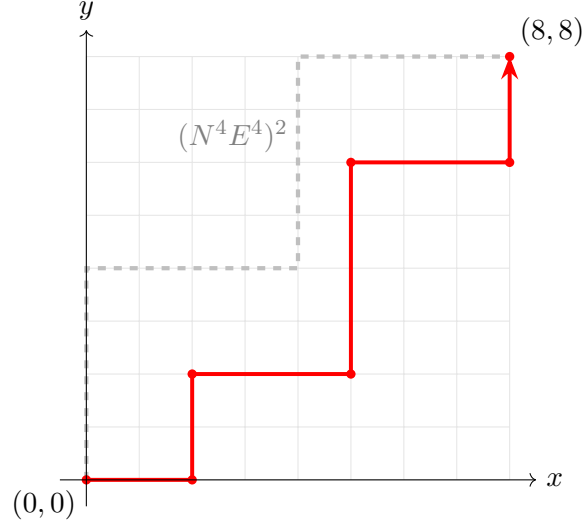
\begin{center}
\begin{tikzpicture}[scale=0.7]
    \draw[step=1, gray!20, thin] (0,0) grid (8,8);
    
    \draw[ultra thick, gray!50, dashed] (0,0) -- (0,4) -- (4,4) -- (4,8) -- (8,8);
    \node[gray, anchor=south east] at (4,6) {$(N^4E^4)^2$};

    \draw[line width=1.5pt, red, -{Stealth[length=3mm]}] 
        (0,0) -- (2,0) -- (2,2) -- (5,2) -- (5,6) -- (8,6) -- (8,8);

    \foreach \coord in {(0,0), (2,0), (2,2), (5,2), (5,6), (8,6), (8,8)} {
        \fill[red] \coord circle (2.5pt);
    }

    \node[below left] at (0,0) {$(0,0)$};
    \node[above right] at (8,8) {$(8,8)$};

    \draw[->] (-0.5,0) -- (8.5,0) node[right] {$x$};
    \draw[->] (0,-0.5) -- (0,8.5) node[above] {$y$};
\end{tikzpicture}
\end{center}

This corresponds to the tableau
\[
    \begin{ytableau}
        1&2&3&4&7&8&12&13\\
        5&6&9&10&14&15&19&20\\
        11&16&17&18&21&22&23&24
    \end{ytableau}
    \]
    which corresponds to the lattice word $111122112231122333223333$. The canon permutations are all $\S_3$ permutations of this lattice word.
\end{example}

The remaining cases are when $j < k < 2(j-1)$ and $j \geq k$. The first case seems complicated, we provide a solution when $n = 2$ that works for any valid values of $j$ and $k$ (see Remark \ref{rem: n=2-cases}). We leave the remaining cases as an open question (see Question \ref{Q: 12j1}):

\begin{proposition} \label{prop: 121-n=2}
    For $2 \le j \le k$,
    \[
    c_2^{k}(\O(12^j1)) = 2\left[\binom{2j-2}{j-1} - \binom{2j-2}{2j-k-3} \right].
    \]
\end{proposition}

\begin{proof}
    This problem is equivalent to counting the number of standard skew tableaux of shape $(k,j-1) \setminus (k-j+1)$. Indeed, given such a tableau, shift all the entries up by $k-j+1$ and add the entries $1,\dots,k-j+1$ into the first $k-j+1$ entries of row one. Then add the entries $k+j,\dots,2k$ as the final missing $k-j+1$ entries in row two. By Lemma \ref{lem: 12j1-rules} with $n = 2$, the conditions are $\pi_{2,j} > \pi_{1,k}$ and $\pi_{2,1} > \pi_{1,k-j+1}$, which force our tableaux to begin and end in this way, so this clearly establishes a bijection.

    We count these directly. Lattice words on two letters correspond to Dyck paths. In this context, Lemma \ref{lem: 12j1-rules} says that the Dyck path must start with at least $k-j+1$ up steps and end with at least $k-j+1$ down steps. If we delete those, we are left with an arbitrary path with $j-1$ up and $j-1$ down steps that never falls below $-(k-j+1)$, reflecting and removing the bad paths gives precisely the stated difference of binomials. The result now follows from Theorem \ref{thm: symmetric}.
    \end{proof}

    \begin{remark} \label{rem: n=2-cases}
        We note that Proposition \ref{prop: 121-n=2} agrees with all of our previous results wherever they apply. Indeed, when $k \ge 2(j-1)$, the second binomial is zero and we recover $2\binom{2j-2}{j-1}$ in agreement with Theorem \ref{thm: 12j1-2(j-1)<k} and with Theorem \ref{thm: 12j1-k=2(j-1)} when $n = 2$. When $j = k$ it evaluates to $\binom{2k-2}{k-1} - \binom{2k-2}{k-3} = C_k = C_2^{(k)}$ in agreement with Theorem \ref{thm: 12j1-j>=k}. 
    \end{remark}

\begin{example}
    For $k = 8$ and $j = 6$, $|\mathcal{L}_2^8(\O(12^61))| = \binom{10}{5} - \binom{10}{1} = 242$ and $c_2^8(\O(12^61)) = 484$.  As one example, the skew tableau
    \[
    \begin{ytableau}
        \none&\none&\none&3&6&7&8&9\\
        1&2&4&5&10&\none&\none&\none
    \end{ytableau}
    \]
    corresponds to the Dyck path UUUDDUDDUUUUDDDD and the tableau
    \[
    \begin{ytableau}
        {\bf 1}&{\bf 2}&{\bf 3}&6&9&10&11&12\\
        4&5&7&8&13&{\bf 14}&{\bf 15}&{\bf 16}
    \end{ytableau}
    \]
    of shape $(8^2)$, which corresponds to the lattice word $1112212211112222$.
\end{example}

Unfortunately, our approach in Proposition \ref{prop: 121-n=2} doesn't generalize because once $n \geq 3$, just considering skew tableaux no longer captures the additional requirements placed on the first and final $k-j+1$ columns on the tableau. We leave this case as an open question.

\begin{question} \label{Q: 12j1}
    For $n \geq 3$ and $4 \le j < k < 2(j-1)$, what is $c_n^k(\O(12^j1))$?  The smallest open case is $j = 4$, $k = 5$, where $|\mathcal{L}_n^5(\O(12^41))|$ begins (for $n \ge 1$)
    \[
    1,\ 19,\ 795,\ 48875,\ 3780362,\ 340746821,\ 34264303289, \dots
    \]
    and $c_n^5(\O(12^41))$ is $n!$ times these values.  It would also be interesting to answer this for $k = j+1$ in general.
\end{question}

We now handle the final case, where $j \geq k$. This case is less interesting because the avoidance condition is irrelevant,

\begin{theorem} \label{thm: 12j1-j>=k}
    For $n \geq 1$ and $k \geq 2$ and $j \geq k$,
    \[
    c_n^k(\O(12^j1)) = n! \, C_n^{(k)}
    \]
    the $k$-dimensional Catalan numbers, which is also the number of standard Young tableaux of shape $(k^n)$.
\end{theorem}

\begin{proof}
    By Theorem \ref{thm: symmetric}, it suffices to prove that $|\mathcal{L}_n^k(\O(12^j1))| = C_n^{(k)}$. This is clear when $j > k$, since neither pattern fits in a word with only $k$ copies of each letter.  When $j = k$ we have $k-j+1 = 1$, so both conditions of Lemma \ref{lem: 12j1-rules} read off relations that hold in every standard Young tableau because the columns increase. Either way $|\mathcal{L}_n^k(\O(12^j1))| = |\mathcal{L}_n^k| = C_n^{(k)}$.
\end{proof}

We note that this provides an additional proof that $c_n^2(1221,2112) = n! C_n$ since $C_n^{(2)} = C_n$. We can ask a similar question for other symmetric sets, like the noncrossing $\{1212,2121\}$ or the more general $\{(12)^j, (21)^j\}$ for $j \ge 2$. This case is surprisingly more difficult than the case we just considered. We can prove some results when $j$ or $n$ are small, but the general cases are left open.

\begin{proposition} \label{prop: 1212-2121}
    For $n \geq 1$, $k \geq 2$,
    \[
    c_n^k(\O(1212)) = n!
    \]
\end{proposition}

\begin{proof}
    To be canon and also noncrossing, we must have $\pi_{i,k} < \pi_{i+1,1}$ for all $1 \leq i \leq n-1$. Otherwise, $\pi_{i,1} < \pi_{i+1,1} < \pi_{i,k} < \pi_{i+1,k}$ which means $(\sigma_\pi)_i (\sigma_\pi)_{i+1} (\sigma_\pi)_i (\sigma_\pi)_{i+1}$ is a subsequence of $\pi$. This is either a $1212$ or a $2121$.

    This means there is only one allowable tableau since all the entries of row $i$ must be less than all the entries in row $i+1$ for $1 \leq i \leq n-1$. Any permutation paired with this tableau will still avoid $1212$ and $2121$. This completes the proof.
\end{proof}

\begin{theorem} \label{thm: 12j-21j}
    For $j \ge 2$ and $k \ge 1$,
    \[
    c_2^k(\O((12)^j)) = 2\sum_{i=1}^{j-1} N(k,i)
    \]
    where $N(k,i)$ is the Narayana number. In particular the sum stabilizes at $2C_k$ once $j > k$.
\end{theorem}

\begin{proof}
    By Theorem \ref{thm: symmetric} this is equivalent to showing that $|\mathcal{L}_2^k(\O((12)^j))|$ is equal to $\sum_{i=1}^{j-1} N(k,i)$. Regular binary lattice words are easily in bijection with Dyck paths, each $1$ corresponds to a $U$ and each $2$ to a $D$. The avoidance condition is equivalent to the Dyck path having at most $j-1$ peaks and at most $j-1$ valleys. Since the number of peaks is always an upper bound for the number of valleys, it suffices to count the number of Dyck paths of semi-length $k$ with at most $j-1$ peaks. This is precisely the stated sum of Narayana numbers.
\end{proof}

In this section so far, we considered some natural generalizations of nonnesting and noncrossing permutations. We now want to consider a different symmetric set that generalizes $\O(121)$.

\begin{theorem} \label{thm: 1a21b-2a12b}
    For $n \ge 1$, $k \ge 2$ and $a,b \ge 1$,
    \[
    c_n^k(\O(1^a21^b)) = n! \; C_n^{(\min(a+b-1,k))}.
    \]
\end{theorem}

\begin{proof}
    For ease of notation, let $\Lambda = \O(1^a21^b) = \{1^a21^b,\, 2^a12^b\}$.  Since $\Lambda$ is symmetric, by Theorem \ref{thm: symmetric} it suffices to show that $|\mathcal{L}_n^k(\Lambda)| = C_n^{(\min(a+b-1,k))}$. If $k < a+b$, neither pattern can occur, since each uses $a+b$ copies of one of its two letters while only $k$ are available.  Hence $|\mathcal{L}_n^k(\Lambda)| = |\mathcal{L}_n^k| = C_n^{(k)}$.
 
    Now suppose $k \ge a+b$ and set $m = k-a-b+2 \ge 2$.  For each row $r$ let
    \[
    B_r = \{\pi_{r,a}, \pi_{r,a+1},\dots,\pi_{r,k-b+1}\},
    \]
    where each $B_r$ contains $m$ entries.  We claim that $\pi$ avoids $\Lambda$ if and only if, for every $r$, no entry of a row other than $r$ lies strictly between $\pi_{r,a}$ and $\pi_{r,k-b+1}$.
 
    An occurrence of $1^a21^b$ consists of $a$ copies of a letter $x$, then one copy of a letter $y > x$, then $b$ further copies of $x$. Let $i$ and $i' > i$ be the rows of $x$ and $y$ respectively.  Such an occurrence exists precisely when some entry $p$ of row $i'$ has at least $a$ entries of row $i$ before it and at least $b$ after it, that is, when $\pi_{i,a} < p < \pi_{i,k-b+1}$. Notice here that the first $a$ and the last $b$ copies of $x$ are disjoint because $a+b \le k$.  This means $1^a21^b$ occurs precisely when some entry of a row $i' > i$ lies strictly between $\pi_{i,a}$ and $\pi_{i,k-b+1}$.  The same argument applied to $2^a12^b$ shows that it occurs precisely when some entry of a row $i < i'$ lies strictly between $\pi_{i',a}$ and $\pi_{i',k-b+1}$.

    Taken together, these conditions force $B_r$ to consist of $m$ consecutive integers. Indeed, if not there would be some $x<y<z$ with $x,z \in B_r$, but $y \not\in B_r$. By the above paragraph, we cannot have $y$ appear in a row less than $r$ without creating a $2^a12^b$ pattern and $y$ cannot appear in a row larger than $r$ without creating a $1^a21^b$ pattern. $y$ also cannot appear in row $r$ outside of $B_r$ because $y < z$ and entries in rows increase. This means the $B_r$ are pairwise disjoint intervals of length $m$ and since entries increase down columns we must have $B_1 < B_2 < \cdots < B_n$.
 
    Now suppose $\pi$ avoids $\Lambda$.  Delete the entries $\pi_{r,a+1},\dots,\pi_{r,k-b+1}$ from each row. That is, delete all but the smallest element of $B_r$, and standardize the remaining $kn - n(m-1) = n(a+b-1)$ entries.  Deleting entries and standardizing preserves relative order, so the result is a standard Young tableau of shape $\big((a+b-1)^n\big)$.
 
    Conversely, let $S$ be a standard Young tableau of shape $\big((a+b-1)^n\big)$ and expand each entry in column $a$ of $S$ into $m$ consecutive integers, shifting every larger entry up by $m-1$ for each such expansion.  The resulting filling of $(k^n)$ has increasing rows by construction.  It is easy to see that its columns increase as well. The two constructions are mutually inverse, so
    \[
    |\mathcal{L}_n^k(\Lambda)| = C_n^{(a+b-1)} = C_n^{(\min(a+b-1,\,k))}. \qedhere
    \]
\end{proof}

This result only relied on the sum $a+b$, so we can immediately conclude,

\begin{corollary}
    If $a+b = a'+b'$, then $\O(1^a21^b)$ and $\O(1^{a'}21^{b'})$ are canon-Wilf equivalent.
\end{corollary}
 
\begin{remark}
    Theorem \ref{thm: 1a21b-2a12b} recovers an earlier result.  Taking $a = b = 1$ gives $\O(121) = \{121,212\}$ and $C_n^{(1)} = 1$, recovering Proposition \ref{prop: 212}.  Taking $a = 1$ and $b=j$ gives $c_n^k(\O(121^j)) = n! C_n^{(\min(j,k))}$. We can also conclude that $c_n^k(\Lambda) = c_n^{k+1}(\Lambda)$ once $k \ge a+b-1$.
\end{remark}

These are only a few examples of symmetric avoidance classes, it would be interesting to find more. We discuss this more in Section \ref{sec: further-study}. The proof of Theorem \ref{thm: 1a21b-2a12b} actually gives us a result about the descent polynomials of this avoidance class. We define and discuss this in the following subsection.

\subsection{Descent Polynomials} \label{subsec: descents}
Recall that $j$ is a \emph{descent} of a word $\pi = \pi_1\cdots\pi_N$ if $\pi_j > \pi_{j+1}$, and let $\des(\pi)$ denote the number of descents of $\pi$.  Write
\[
D_n^k(x) := \sum_{\pi \in \C_n^k} x^{\des(\pi)}
\]
for the descent polynomial of all canon permutations.  Elizalde \cite{E23,E25} proved that $D_n^k(x)$ has a product structure, and in particular that $D_n^2(x) = A_n(x)N_n(x)$, where $A_n$ is an Eulerian polynomial and $N_n$ a Narayana polynomial. We note that our notation differs from Elizalde's, we use $D_n^k(x)$ instead of $C_n^k(x)$ to avoid confusion with our other notation. 

Beck and Deligeorgaki \cite{BD24} recovered and generalized this from Stanley's theory of $(P,\omega)$-partitions \cite{Sta11,Sta72}, showing that $D_n^k(x) = A_n(x)\,h^*_{[k]\times[n]}(x)$, where $h^*_{[k]\times[n]}(x)$ is the numerator of the generating function of the order polynomial of $[k] \times [n]$, see \cite{BD24} for details. This implies that $D_n^k(x)$ is palindromic and $\gamma$-positive. Here palindromic means the coefficient sequence is symmetric, and $\gamma$-positive is a stronger condition on the distribution of the coefficients: a palindromic polynomial is $\gamma$-positive if its coefficients are non-negative when expressed in the $\gamma$-basis $\{x^i(x+1)^{d-2i} \; | \; 0 \le i \le \lfloor d/2 \rfloor\}$, where $d$ is the degree.  The bijection of Theorem \ref{thm: 1a21b-2a12b} turns out to respect descents, so these properties are inherited by the avoidance classes $\C_n^k(\O(1^a21^b))$. This means,
 
\begin{theorem} \label{thm: 1a21b-descents}
    Let $n \ge 1$, $k \ge 2$ and $a,b \ge 1$, and set $\ell = \min(a+b-1,\,k)$.  Then
    \[
    \sum_{\pi \in \C_n^k(\O(1^a21^b))} x^{\des(\pi)} = D_n^{\ell}(x).
    \]
    In particular this descent polynomial is palindromic and $\gamma$-positive.
\end{theorem}
 
\begin{proof}
    If $k < a+b$ then $\ell = k$ and, as in the first paragraph of the proof of Theorem \ref{thm: 1a21b-2a12b}, neither pattern occurs in a $k$-regular word, so $\C_n^k(\O(1^a21^b)) = \C_n^k$ and there is nothing to prove.
 
    Suppose $k \ge a+b$ and set $m = k-a-b+2$ as before.  The proof of Theorem \ref{thm: 1a21b-2a12b} gives a bijection 
    \[
    \pi = (\sigma_\pi, T_\pi) \mapsto \pi' = (\sigma_\pi, S)
    \]
    from $\C_n^k(\O(1^a21^b))$ to $\C_n^{a+b-1}$, where $S$ is obtained from $T_\pi$ by deleting the entries $\pi_{r,a+1},\dots,\pi_{r,k-b+1}$ of each row $r$ and standardizing.  We claim it preserves descents.
 
    Fix a row $r$.  By Theorem \ref{thm: 1a21b-2a12b}, the block $B_r = \{\pi_{r,a},\dots,\pi_{r,k-b+1}\}$ consists of $m$ consecutive integers. This means in the word $\pi$ these are $m$ consecutive occurrences of $(\sigma_\pi)_r$.  Passing from $\pi$ to $\pi'$ deletes all but the first of these $m$ occurrences.  This does not create or remove a descent because we are removing consecutive repetitions of an element. This immediately implies that $\des(\pi') = \des(\pi)$.  Applying this to each row shows that
    \[
    \sum_{\pi \in \C_n^k(\O(1^a21^b))} x^{\des(\pi)} = \sum_{\pi' \in \C_n^{a+b-1}} x^{\des(\pi')} = D_n^{a+b-1}(x). \qedhere
    \]
\end{proof}
 
\begin{remark}
    Theorem \ref{thm: 1a21b-descents} says that $\C_n^k(\O(1^a21^b))$ is not just equinumerous with the smaller class $\C_n^{\ell}$, but equidistributed with it by descents.  Combined with Theorem \ref{thm: 1a21b-2a12b} this recovers $c_n^k(\O(1^a21^b)) = D_n^{\ell}(1) = n!\,C_n^{(\ell)}$. At the two smallest values of $\ell$, this specializes to familiar polynomials. When $\ell = 1$, which by Proposition \ref{prop: 212} is the case $a = b = 1$, we get $D_n^1(x) = A_n(x)$, the Eulerian polynomial, and when $\ell = 2$, we get $A_n(x)N_n(x)$.
\end{remark}

\section{Further Directions} \label{sec: further-study}

The results in each of our sections generally rest on one (or more) of three hypotheses: that the underlying shape of the tableau is the rectangle $(k^n)$, that the forbidden set $\Lambda$ is symmetric, and that $n$ is the parameter that varies. Relaxing each suggests a different direction for future work.

\subsection{Changing the tableau shape}

The framework of Section \ref{sec: background} generalizes immediately if we allow $T_\pi$ to have a shape other than $(k^n)$. This amounts to dropping the $k$-regularity assumption and would mean we no longer have a canon permutation. Some natural options for tableau shapes are the elongated staircase $(k,k-1,k-2,\dots,1^{n-k})$, or a skew shape $(k^n) \setminus (j^m)$, where the canon permutations are precisely the case $j = m = 0$. In the language of Remark \ref{rem: poset-contrast}, this changes the poset $[k] \times [n]$ itself, whereas the avoidance conditions of Section \ref{sec: symmetry} add relations and the dissonant canon permutations of \cite{BD24} delete them. Avoidance in lattice words was studied for $n = 2$ \cite{PS26}, but already for $n \ge 3$, and for shapes other than the rectangle, little seems to be known.

\subsection{Changing the patterns}

Section \ref{sec: symmetry} focused on avoiding symmetric sets, allowing us to ignore $\sigma_\pi$ and just consider avoidance in lattice words. This was motivated by approachability and the fact that the noncrossing and nonnesting avoidance sets are symmetric. Since every orbit $\O(\lambda)$ is symmetric, there are many more to consider, and several appear experimentally to have clean enumerations. For example,

\begin{conjecture} \label{conj: 1213}
    For $n \ge 2$ and $k \ge 2$,
    \[
    c_n^k(\O(1213)) = n! \; C_k,
    \]
    where $C_k$ is the $k$-th Catalan number.  In particular the lattice word enumeration does not depend on $n$.
\end{conjecture}

Recall $\O(1213) = \{1213,1312,2123,2321,3132,3231\}$. Since $\O(1213)^r = \O(1232)$, the same formula should hold for $\O(1232)$. Another example with a nice enumeration is

\begin{conjecture} \label{conj: 12112}
    For $n \ge 1$ and $k \ge 2$,
    \[
    c_n^k(\O(12112)) = n! \; \C_{k,n}.
    \]
\end{conjecture}

It is interesting that the Fuss-Catalan numbers arise here as well. Theorem \ref{thm: 112} produces them from a single non-symmetric pattern of length three, while Conjecture \ref{conj: 12112} produces them from a symmetric pair of patterns of length five. By Remark \ref{rem: orbits}, $\O(12212) = \O(12112)^r$ satisfies the same conjecture. One other interesting symmetric set is $\O(1231)$.

\begin{conjecture} \label{conj: 1231}
    Let $a_n = |\mathcal{L}_n^k(\O(1231))|$. Then $a_1 = 1$, $a_2 = C_k$, and for $n \ge 3$ and $k \ge 1$,
    \[
    a_n = a_{n-1} + (C_k - 1) \, a_{n-2}.
    \]
\end{conjecture}

Here $\O(1231) = \{1231,1321,2132,2312,3123,3213\}$. The value $a_2 = C_k$ is forced, since a pattern on three letters cannot occur in a word on two letters. Three special cases of this conjecture are worth mentioning: 

\begin{itemize}
    \item When $k = 2$ we have $C_k - 1 = 1$. So our conjecture becomes $c_n^2(\O(1231)) = n! \, F_{n+1}$, where $F_n$ is the $n$-th Fibonacci number.
    \item When $k = 3$, the recurrence $a_n = a_{n-1} + 4a_{n-2}$ is \cite[A006131]{OEIS}. 
    \item If we take $n = 3$, we get $|\mathcal{L}_3^k(\O(1231))| = 2C_k -1$ for every $k \ge 1$.
\end{itemize}  

Beyond finding more such sets, it would be interesting to understand which structural properties survive adding avoidance conditions. Subsection \ref{subsec: descents} suggests one question in this direction,

\begin{question} \label{Q: palindromic-sym}
    For which symmetric sets $\Lambda$ is the descent polynomial of $\C_n^k(\Lambda)$ palindromic, or $\gamma$-positive?  By Theorem \ref{thm: symmetric} such a class is a disjoint union over $\sigma \in \S_n$ of sets of linear extensions of a labeled poset obtained from $[k]\times[n]$ by adjoining relations that the avoidance conditions impose. This means the machinery of \cite{BD24} applies whenever the resulting poset is again a product.  Theorem \ref{thm: 1a21b-descents} is one example of when this happens.
\end{question}

\subsection{Changing which parameter varies}

Conjecture \ref{conj: 1213} is constant in $n$, and the last special case of Conjecture \ref{conj: 1231} is a clean formula in $k$ for fixed $n = 3$. Both suggest fixing $n$ and letting $k$ vary, which is the opposite of the convention thus far. Most work to date, including the still open enumeration of $c_n^2(321) = c_n^2(123)$, fixes $k = 2$ and varies $n$. From this perspective even the non-symmetric sets look tractable.

\begin{question}
    What is $c_3^k(\Lambda)$ for $\Lambda \subseteq \S_3$?
\end{question}

For example, it seems like

\begin{conjecture}
    For $k \geq 1$, $c_3^k(213,312) = 4 C_k$ where $C_k$ is the $k$-th Catalan number.
\end{conjecture}

\noindent {\tiny \emph{The views expressed in this paper are those of the authors and do not reflect the official policy or position of the U.S. Naval Academy, Department of the Navy, the Department of Defense, or the U.S. Government.} \par}

\end{document}